\documentclass[twocolumn,nameyear,amsthm]{autart}

\usepackage[T1]{fontenc}
\usepackage{lmodern}
\usepackage{amsmath,amssymb,amsfonts,mathrsfs,mathtools}
\usepackage{booktabs}
\usepackage[longnamesfirst]{natbib}

\allowdisplaybreaks

\newcommand{\R}{\mathbb R}

\newcommand{\Kfb}{\mathsf K}

\newcommand{\Fbar}{F}
\newcommand{\B}{\mathcal B}
\newcommand{\Bzero}{\mathsf B_0}
\newcommand{\BT}{\mathsf B_T}
\newcommand{\bd}{\mathrm{bd}}
\newcommand{\lin}{\mathrm{lin}}
\newcommand{\Jc}{\mathbb J}
\newcommand{\xinf}{x_\infty}
\newcommand{\uinf}{u_\infty}
\newcommand{\pinf}{p_\infty}
\newcommand{\norm}[1]{\lVert #1\rVert}
\newcommand{\ip}[2]{\langle #1,#2\rangle}

\numberwithin{equation}{section}
\newtheorem{theorem}{Theorem}[section]
\newtheorem{proposition}[theorem]{Proposition}
\newtheorem{lemma}[theorem]{Lemma}
\newtheorem{corollary}[theorem]{Corollary}
\newtheorem{definition}[theorem]{Definition}

\begin{document}
\begin{frontmatter}
\date{17 August 2026}
\title{Horizon-Uniform Sensitivity and Decay of Terminal Reward Perturbations in Discrete-Time Pontryagin Systems\thanksref{footnoteinfo}}
\thanks[footnoteinfo]{Corresponding author: Pyuyi Chufeng Huang.}

\author[scu1]{Pyuyi Chufeng Huang}\ead{pyuyi233@gmail.com}
\author[scu1]{Zikang Song}\ead{1289094751@qq.com}
%\author[scu1]{Xingshu Chen}\ead{chenxsh@scu.edu.cn}

%\address[scu1]{School of Cyber Science and Engineering, Sichuan University, Chengdu, Sichuan, China}
\address[scu1]{School of Mathematics, Sichuan University, Chengdu, Sichuan, China}


\begin{abstract}
We study local stationary solutions of finite-horizon discrete-time
Pontryagin systems near a steady extremal.  Suppose that the stationarity
equation for the control is regular, the reduced state--costate map is
hyperbolic, and the endpoint conditions satisfy a scaled transversality
condition with respect to the stable and unstable subspaces.  Then the
linearized boundary-value problem admits an inverse whose Green estimate is
uniform in the horizon.  The Green kernel separates interior decay from the
two reflections induced by the endpoint conditions.  For $x_0=x_{\rm in}$ and
$p_T=r_x(x_T,y)$, a contraction argument in a weighted norm proves existence
and uniqueness in a neighborhood independent of $T$, together with uniform
Lipschitz estimates and a pointwise quadratic remainder.  We also derive an
explicit admissible data radius and an a posteriori criterion for existence
and local uniqueness near an approximate trajectory.  For these graph
boundary conditions, a one-sided Green estimate shows that a perturbation of
the terminal reward changes the initial control and the gradient with respect to the initial state
of the stationary objective by $O(e^{-\alpha_{\rm ter} T})$ for every
$\alpha_{\rm ter}$ below the dichotomy rate.  For linear-quadratic systems with
invertible $A$, stabilizable $(A,B)$, $Q\succ0$, $R\succ0$, and a nonpositive terminal
Hessian, a symplectic graph condition verifies the assumptions, and the
finite-horizon Riccati matrix and initial feedback gain converge at rate
$O(e^{-2\gamma T})$.  Numerical experiments verify the certificates and the
predicted decay rates.
\end{abstract}
\end{frontmatter}
\section{Introduction}\label{sec:intro}
Finite-horizon optimal control methods repeatedly solve two-point Pontryagin
equations, for example in shooting, continuation, and receding-horizon
computations \citep{bertsekas2017dynamic,grune2017nmpc}.
A fixed-horizon implicit function theorem gives a local stationary branch, but
does not control the inverse uniformly as the horizon grows.

Dichotomy theory relates boundary
transversality to well-conditioned linear boundary-value problems
\citep{coppel1978dichotomies}.
Implicit-function arguments in weighted sequence spaces yield horizon-uniform nonlinear
sensitivity once a uniformly bounded linear inverse is available
\citep{grune2021abstract}.  Hyperbolic Hamiltonian equilibria and transverse
stable and unstable manifolds also yield two-sided turnpike estimates.
For optimal solutions and KKT systems, uniform
exponential sensitivity has been proved under second-order sufficient
conditions, constraint qualifications, or controllability assumptions.
Dynamic programming gives the complementary value-function description of
the same optimal trajectories; its differentiable characteristics are the
Hamiltonian state--costate equations
\citep{bertsekas2017dynamic,sassano2025infinite}.
The correspondence is classical; here it determines when the stationary
quantities certified below coincide with value gradients and feedback laws.

First, \citet{na2020exponential} prove decay of directional sensitivities for
locally optimal discrete-time programs under uniform second-order sufficiency,
controllability, and derivative bounds, while
\citet{shin2021controllability} relate second-order sufficiency and constraint
regularity to controllability and observability; we instead use a
verifiable transversality condition for a hyperbolic state--costate map to
construct a local stationary Pontryagin branch without assuming optimality.
Second, \citet{shin2022exponential} prove that primal--dual sensitivity in
graph-structured nonlinear programs decays with graph distance under a strong second-order sufficient
condition and the linear independence constraint qualification; for the
temporal chain, we separate the contributions of both boundaries and use a
scaled boundary matrix to bound the influence of the terminal condition on the
initial control and initial costate.  Third, \citet{breiten2020turnpike} and
\citet{kunisch2020terminal} prove exponential receding-horizon estimates for
infinite-dimensional linear-quadratic (LQ) problems and stabilization problems
with terminal penalties, respectively; we instead bound terminal reward perturbations in a
discrete-time nonlinear Pontryagin boundary problem, with optimality required
only for its interpretation in optimal control.  Fourth,
\citet{sakamoto2021turnpike} derive nonlinear turnpike estimates from stable and
unstable manifolds, while \citet{xu2018exponentially} prove exponential
accuracy of an overlapping temporal decomposition in the overlap size; our
framework combines the scaled boundary matrix with a horizon-independent
nonlinear existence radius and pointwise quadratic remainder, adds a separate
a posteriori existence test, and uses the LQ M\"obius representation to obtain
the $O(e^{-2\gamma T})$ bound on sensitivity to the terminal Hessian.

We seek a finite-horizon matrix criterion that controls both endpoint
reflections and remains valid for the nonlinear terminal condition
$p_T=r_x(x_T,y)$.  Our contributions are as follows.  First, we derive a Green
estimate from the scaled matrices $\widehat\Gamma_T$; its kernel separates
interior decay from two endpoint reflections.  Second, we prove the existence
of a horizon-uniform local stationary branch, a pointwise remainder of order
$O(w_T(t)^2\|\Delta\|^2)$, and a uniform Lipschitz estimate for the initial
control.  We also derive explicit a priori and a posteriori radii for existence
and local uniqueness.  Third, for graph boundary conditions, we isolate the
contribution of the terminal boundary and prove that perturbations of the
terminal reward have exponentially small effects on controls near the initial
time and on the gradient of the stationary objective with respect to the
initial state.  Fourth,
for LQ systems with invertible $A$, stabilizable $(A,B)$, $Q\succ0$, $R\succ0$, and
$S\preceq0$, we reduce the matrix criterion to symplectic graph conditions,
prove invertibility for every horizon, and obtain an explicit M\"obius formula
and $O(e^{-2\gamma T})$ convergence of the Riccati matrices and feedback gains.

The main results establish stationary extremals.  An additional second-order
sufficient condition is required to identify them as local optima.
Section~\ref{sec:model} gives the reduction,
Section~\ref{sec:inverse} proves the linear estimate,
Sections~\ref{sec:reconstruction}--\ref{sec:branches} give the nonlinear
results, decay under terminal reward perturbations, and computable certificates,
Section~\ref{sec:lq} treats the LQ
verification, and Section~\ref{sec:numerics} reports the numerical tests.

\section{Model, reduction, and standing assumptions}\label{sec:model}
For a horizon $T\in\mathbb N$, the state is in $\R^d$, the control set $U\subset\R^m$ is compact and convex, and
\begin{equation}\label{eq:dynamics}
      x_{t+1}=\Fbar(x_t,u_t),\qquad 0\le t<T.
\end{equation}
For $y$ in an open set $Y\subset\R^q$, the maximization objective and Hamiltonian are
\begin{equation}\label{eq:objective}
      J_T(u)=r(x_T,y)-\lambda\sum_{t=0}^{T-1}\ell(x_t,u_t),\qquad \lambda>0,
\end{equation}
\begin{equation}\label{eq:ham}
      H(x,u,p)=-\lambda\ell(x,u)+\ip{p}{\Fbar(x,u)}.
\end{equation}

All local branch estimates use fixed neighborhoods, independent of $T$.
The maps $\Fbar$ and $\ell$ are $C^3$ near the reference, and terminal
rewards $r(\cdot,y)$ are $C^3$ in $x$ when nonlinear terminal conditions are
used.  The required uniform derivative bounds are stated below.

An interior stationary reference satisfies
\begin{equation}\label{eq:stationary}
\begin{aligned}
   \xinf=\Fbar(\xinf,\uinf),\qquad
   \pinf&=H_x(\xinf,\uinf,\pinf),\\
   H_u(\xinf,\uinf,\pinf)&=0,
\end{aligned}
\end{equation}
with $\uinf\in\operatorname{int}U$.

\subsection{Pontryagin equations and sign convention}\label{sec:pmp}
We use the maximization convention for \eqref{eq:objective} and the forward
Hamiltonian \eqref{eq:ham}; the costate paired with the step from $t$ to
$t+1$ is $p_{t+1}$.  Consequently
\[
\begin{aligned}
   H_x(x,u,p_+)&=\partial_x\Fbar(x,u)^\top p_+-\lambda\partial_x\ell(x,u),\\
   H_u(x,u,p_+)&=\partial_u\Fbar(x,u)^\top p_+-\lambda\partial_u\ell(x,u).
\end{aligned}
\]
If a trajectory is locally optimal for \eqref{eq:objective}, one-sided
variations of a single control value in the convex set $U$ give
\begin{equation}\label{eq:pmp-system}
\begin{gathered}
    x_{t+1}=\Fbar(x_t,u_t),\qquad
    p_t=H_x(x_t,u_t,p_{t+1}),\\
    p_T=\partial_x r(x_T,y),\\
    \langle H_u(x_t,u_t,p_{t+1}),v-u_t\rangle\le0,\\
    0\le t<T,\qquad v\in U.
\end{gathered}
\end{equation}
These are the standard discrete-time Pontryagin conditions
\citep{halkin1966maximum,bertsekas2017dynamic}; the signs reflect maximization
of terminal reward minus accumulated cost.

All sufficiently small branches considered below remain in
$\operatorname{int}U$.  On such branches the variational inequality is
equivalent to $H_u=0$.  The local stationarity graph constructed next therefore
eliminates $u_t$ as $u_t=\nu(x_t,p_{t+1})$ and turns the first two equations in
\eqref{eq:pmp-system} into the reduced implicit equation \eqref{eq:reduced}.
The endpoint rows are retained rather than eliminated.  Thus a stationary
Pontryagin branch in this paper means a local branch of \eqref{eq:pmp-system}
after interior stationarity reduction.

\subsection{Relation to Bellman recursion}\label{sec:bellman}
For an optimal branch, define
\begin{align*}
 V_{t,T}(x;y)&=\max_{u_t,\ldots,u_{T-1}}
 \left\{r(x_T,y)-\lambda\sum_{k=t}^{T-1}\ell(x_k,u_k)\right\},\\
 Q_{t,T}(x,u;y)&=-\lambda\ell(x,u)
 +V_{t+1,T}(\Fbar(x,u);y),
\end{align*}
with $V_{T,T}=r$.  If the relevant value functions are differentiable and
the Bellman maximizers are interior, then, along the optimal trajectory,
\begin{equation}\label{eq:bellman-pmp}
 \begin{aligned}
 p_{t+1}&=D_xV_{t+1,T}(x_{t+1};y),\\
 D_uQ_{t,T}&=H_u=0,\qquad D_xV_{t,T}=H_x=p_t .
 \end{aligned}
\end{equation}
These identities follow from the chain rule and the envelope theorem.  Thus
the costate equals the state gradient of the future value, and the Pontryagin
equations are the first-order characteristic equations associated with the
differentiable Bellman recursion
\citep{bertsekas2017dynamic,sassano2025infinite}.

The Hamiltonian and the Bellman $Q$-function are nevertheless different: in
general $H(x,u,p_+)\ne Q_{t,T}(x,u;y)$.  If $V_{t+1,T}$ is twice differentiable,
set $x_+=\Fbar(x,u)$ and $p_+=D_xV_{t+1,T}(x_+;y)$.  Then
\begin{equation}\label{eq:bellman-hessian}
 D_{uu}^2Q_{t,T}=H_{uu}
 +\Fbar_u^\top D_{xx}^2V_{t+1,T}\Fbar_u .
\end{equation}
Consequently, \eqref{eq:regular-stationarity} below guarantees only that the
stationarity equation can be solved locally for $u$ at fixed $p_+$.  It does
not imply that $u$ maximizes the Bellman $Q$-function or that the resulting
branch is optimal.

\medskip\noindent\textit{Local control elimination.}
Assume that $\Fbar$ and $\ell$ are $C^3$ near the reference and
\begin{equation}\label{eq:regular-stationarity}
 H_u(\xinf,\uinf,\pinf)=0,\qquad
 \det H_{uu}(\xinf,\uinf,\pinf)\ne0 .
\end{equation}
Put $\zeta_\infty=(\xinf,\pinf)$.  The finite-dimensional implicit function
theorem gives neighborhoods $\Omega$ of $\zeta_\infty$ and
$V\Subset\operatorname{int}U$ of $\uinf$,
and a unique $C^2$ map $\nu:\Omega\to V$ such that
\begin{equation}\label{eq:legendre-graph}
 H_u(x,\nu(x,p_+),p_+)=0,\qquad
 \nu(\xinf,\pinf)=\uinf .
\end{equation}
After shrinking the neighborhoods, $D\nu$ and $D^2\nu$ are bounded and
\begin{equation}\label{eq:nu-derivative}
 D\nu=-H_{uu}^{-1}(H_{ux},H_{up}) .
\end{equation}
These are standard consequences of the implicit function theorem; no local
maximum property is used.

With $z_t=(x_t,p_t)$, the reduced state-costate step is
\begin{equation}\label{eq:reduced}
      \mathcal F(z_t,z_{t+1})=0,
\end{equation}
where
\begin{equation}\label{eq:Fdef}
\mathcal F(z_t,z_{t+1})=
\binom{x_{t+1}-\Fbar(x_t,\nu(x_t,p_{t+1}))}
      {p_t-H_x(x_t,\nu(x_t,p_{t+1}),p_{t+1})}.
\end{equation}
Let $\nu_x=D_x\nu(\xinf,\pinf)$ and $\nu_p=D_{p_+}\nu(\xinf,\pinf)$.  In \eqref{eq:reduced-linearization-blocks}, $\Fbar_x,\Fbar_u$ are the derivatives of the original dynamics $\Fbar$ at $(\xinf,\uinf)$ and all Hamiltonian derivatives are evaluated at $(\xinf,\uinf,\pinf)$ and $H_{x p_+}$ denotes the derivative of $H_x$ with respect to the forward costate argument.  For $h=(\xi,\pi)$ and $h^+=(\xi^+,\pi^+)$, the linearization is $L_0h+L_1h^+$ with
\begin{equation}\label{eq:reduced-linearization-blocks}
\begin{aligned}
L_0h&=
\binom{-(\Fbar_x+\Fbar_u\nu_x)\xi}
      {\pi-(H_{xx}+H_{xu}\nu_x)\xi},\\
L_1h^+&=
\binom{\xi^+-\Fbar_u\nu_p\pi^+}
      {-(H_{x p_+}+H_{xu}\nu_p)\pi^+}.
\end{aligned}
\end{equation}

\medskip\noindent\textit{Uniform local assumptions.}
Given a product neighborhood $Z=Z_x\times Z_p\subset\R^d\times\R^d$ of $0$ contained in $\Omega-\zeta_\infty$, set
\begin{equation}\label{eq:tildeF-def}
\begin{aligned}
      \widetilde{\mathcal F}(z,z^+)
      &:=\mathcal F\bigl((\xinf,\pinf)+z,\;(\xinf,\pinf)+z^+\bigr),\\
      &\hspace{1.4em} z,z^+\in Z .
\end{aligned}
\end{equation}
The reduced state--costate step \eqref{eq:reduced} on
$(\xinf,\pinf)+Z$ is $\widetilde{\mathcal F}(z,z^+)=0$.  We fix $Z$,
a convex $U_0\Subset\operatorname{int}U$, and compact product balls
$B_x\Subset Z_x$, $B_p\Subset Z_p$ such that $\nu$ maps all mixed
arguments $(\xinf+\xi,\pinf+\pi^+)$ into $U_0$ and, for some constants
$M_{\rm int},M_{\rm ter}$,
\[
\begin{aligned}
      \sup_{(\xinf,\pinf)+Z}\|D\nu\|
      +\sup_{(\xinf,\pinf)+Z}\|D^2\nu\|
      &\le M_{\rm int},\\
      \sup_{(z,z^+)\in(B_x\times B_p)\times(B_x\times B_p)}
      \|D^2\widetilde{\mathcal F}(z,z^+)\|
      &\le M_{\rm int}.
\end{aligned}
\]
For a reference $y_\infty\in Y$, assume that a neighborhood
$Y_0\Subset Y$ satisfies
\[
       \sup_{x\in \xinf+B_x,\ y\in Y_0,\ 1\le j\le3}
       \|D_x^j r(x,y)\|\le M_{\rm ter}.
\]
The maps $y\mapsto r_x(x,y)$ and $y\mapsto r_{xx}(x,y)$ are assumed
continuous uniformly in this tube, and locally Lipschitz uniformly in $x$
when joint Lipschitz dependence on $(x_{\rm in},y)$ is asserted.  All radii
below are chosen inside this fixed tube.  Finally, direct verification from problem data
uses
\begin{equation}\label{eq:explicit-transition}
       \det L_1\ne0,\qquad M:=-L_1^{-1}L_0 .
\end{equation}
Thus the linearized interior equation is
$h_{t+1}=Mh_t+L_1^{-1}f_t$.  The alternative abstract formulation only
requires the Green property stated next.

\section{Linear endpoint estimates}\label{sec:inverse}
The nonlinear argument requires the following uniform inverse estimate.
\begin{definition}[Uniform Green property]\label{cond:green}
Let $L:\mathbb R^n\times\mathbb R^n\to\mathbb R^n$ be linear and let $\Bzero,\BT\in\mathbb R^{n\times n}$ be fixed boundary-row matrices, not necessarily invertible individually.  Given constants $C_G\ge1$, $\gamma>0$, and $T_0\ge1$, the triple $(L,\Bzero,\BT)$ has the uniform Green property with these constants if, for every $T\ge T_0$, every interior forcing $f=(f_0,\ldots,f_{T-1})$, and every boundary datum $b\in\R^n$, the following two conditions hold:
\begin{enumerate}
\item[(WP)] The linear problem
\begin{equation}\label{eq:linear-bvp}
\begin{aligned}
      L[h_t,h_{t+1}]&=f_t,\qquad 0\le t<T,\\
      \Bzero h_0+\BT h_T&=b.
\end{aligned}
\end{equation}
has a unique solution $h=(h_0,\ldots,h_T)$.
\item[(G)] With fixed finite-dimensional norms and
\begin{equation}\label{eq:kappa-kernel}
\begin{aligned}
      \kappa_T(t,s)=e^{-\gamma|t-s|}
      +e^{-\gamma(t+T-s)}+&e^{-\gamma(T-t+s)},\\ &0\le s<T,
\end{aligned}
\end{equation}
the solution from \emph{(WP)} satisfies
\begin{equation}\label{eq:green-estimate}
\begin{aligned}
      \norm{h_t}
      &\le C_G\Bigl[(e^{-\gamma t}+e^{-\gamma(T-t)})\norm b\\
      &\hspace{4.0em}
      +\sum_{s=0}^{T-1}\kappa_T(t,s)\norm{f_s}\Bigr]
\end{aligned}
\end{equation}
for all $0\le t\le T$.
\end{enumerate}
When only the property is named, the constants are independent of $T$.
\end{definition}

The last two terms in \eqref{eq:kappa-kernel} are the finite-horizon boundary
reflections.  We write $\mathcal L_T^{\mathrm{bd}}b$ and
$\mathcal L_T^{\mathrm{for}}f$ for the boundary and forcing solutions.

A concrete verification comes from the explicit transition
\eqref{eq:explicit-transition} and a uniformly invertible scaled boundary
matrix.

\begin{theorem}[Finite-horizon Green certificate]\label{prop:scaled-green}
Consider, for a horizon $N$,
\begin{equation}\label{eq:explicit-linear-bvp}
       h_{t+1}=Mh_t+g_t,\quad 0\le t<N,\qquad \Bzero h_0+\BT h_N=b,
\end{equation}
on $\R^n$.  Suppose $M$ is hyperbolic and let the columns of $V_s$ and $V_u$ be bases of its stable and unstable subspaces, with $\dim E_s+\dim E_u=n$; write $MV_s=V_sD_s$ and $MV_u=V_uD_u$, and assume that, for some $K\ge1$ and $\gamma>0$,
$\norm{V_sD_s^k}\le K e^{-\gamma k}$ and $\norm{V_uD_u^{-k}}\le K e^{-\gamma k}$ for all $k\ge0$.
For each horizon define the scaled boundary matrix
\begin{equation}\label{eq:scaled-Gamma-general}
       \widehat\Gamma_N=
       \Bigl[
       \Bzero V_s+\BT V_sD_s^{\,N}\quad
       \Bzero V_uD_u^{-N}+\BT V_u
       \Bigr],
\end{equation}
where the superscripts are powers, not transposes.  If $\widehat\Gamma_N$ is invertible for every $N\ge T_0$ and $\sup_{N\ge T_0}\norm{\widehat\Gamma_N^{-1}}<\infty$, then the uniform Green property of Definition~\ref{cond:green} holds for \eqref{eq:explicit-linear-bvp}, with constants depending only on the dichotomy bounds, the spectral-projector norms, the fixed boundary matrices, and the displayed uniform bound for $\widehat\Gamma_N^{-1}$.  Moreover, the uniform inverse bound follows if
\begin{equation}\label{eq:scaled-Gamma-limit}
       \widehat\Gamma_\infty:=\bigl[\Bzero V_s\quad \BT V_u\bigr]
\end{equation}
is invertible and, for some cutoff $N_*$ large enough for the Neumann argument, each matrix $\widehat\Gamma_N$ with $T_0\le N<N_*$ is invertible.
\end{theorem}

\noindent We say that \emph{scaled boundary transversality} holds if $\widehat\Gamma_N$ in \eqref{eq:scaled-Gamma-general} is invertible for every $N\ge T_0$ with $\sup_{N\ge T_0}\|\widehat\Gamma_N^{-1}\|<\infty$.  The bases are fixed once and for all: replacing them by a fixed pair $V_sA_s,V_uA_u$ right-multiplies $\widehat\Gamma_N$ by $\operatorname{diag}(A_s,A_u)$, so the property is basis independent up to condition numbers, while no $N$-dependent rescaling is allowed.

\begin{proof}
Let $V=[V_s\ V_u]$, write $V^{-1}=\binom{W_s}{W_u}$, and set
$\Pi_s=V_sW_s$, $\Pi_u=V_uW_u$.  Every homogeneous solution has the unique
form
\[
 q_t=V_sD_s^ta+V_uD_u^{t-N}\eta .
\]
Its boundary equation is
$\widehat\Gamma_N\binom{a}{\eta}=b$, so uniform invertibility gives
\begin{equation}\label{eq:homogeneous-green}
 \|q_t\|\le C(e^{-\gamma t}+e^{-\gamma(N-t)})\|b\|.
\end{equation}

For forcing, define
\[
 p_t=\sum_{j<t}M^{t-1-j}\Pi_sg_j
      -\sum_{j\ge t}(M|_{E_u})^{t-1-j}\Pi_ug_j .
\]
Then $p_{t+1}-Mp_t=g_t$ and
\begin{align}
 \|p_t\|&\le C\sum_j e^{-\gamma|t-j|}\|g_j\|,\label{eq:particular-bound}\\
 \|p_0\|+\|p_N\|&\le C\sum_j
 (e^{-\gamma j}+e^{-\gamma(N-j)})\|g_j\|.\label{eq:endpoint-particular}
\end{align}
Correct the boundary residual
$b-\Bzero p_0-\BT p_N$ by a homogeneous solution.  Multiplying
\eqref{eq:endpoint-particular} by the two factors in
\eqref{eq:homogeneous-green} produces four terms.  The products
$e^{-\gamma(t+j)}$ and $e^{-\gamma(2N-t-j)}$ are bounded by
$e^{-\gamma|t-j|}$; the other two are
$e^{-\gamma(t+N-j)}$ and $e^{-\gamma(N-t+j)}$.  Together with
\eqref{eq:particular-bound}, these are exactly the terms in
$\kappa_N(t,j)$.  This proves \eqref{eq:green-estimate}.  With zero forcing
and boundary datum, invertibility of $\widehat\Gamma_N$ gives uniqueness.

Finally,
$\|\widehat\Gamma_N-\widehat\Gamma_\infty\|\le Ce^{-\gamma N}$.
Neumann's lemma covers all sufficiently large $N$, and the finitely many
remaining inverses give one uniform bound.
\end{proof}

This scaled boundary transversality property is also open: under sufficiently
small continuous perturbations of $M,\Bzero,\BT$, the
Riesz projectors and local bases vary continuously, a common spectral gap
persists, and $\widehat\Gamma_N\to\widehat\Gamma_\infty$ uniformly.  Neumann's
lemma covers large $N$, while continuity of the finitely many remaining least
singular values covers short horizons.  Any exponential rate below the stable
and unstable spectral separation is admissible after increasing $K$.

The two-sided kernel above is needed for general endpoint data.  For Pontryagin
graph rows, terminal-only data admit a sharper, one-sided estimate.
Write $P_x=(I\;0)$ and $R_S=(-S\;I)$ on $\R^d\times\R^d$.

\begin{lemma}[One-sided terminal Green estimate]\label{lem:right-green}
In Theorem~\ref{prop:scaled-green}, let $n=2d$, $\dim E_s=\dim E_u=d$, and
write $V_s=\binom{X_s}{P_s}$ and $V_u=\binom{X_u}{P_u}$.  Let $\mathcal S$
be a bounded set of matrices for which $X_s$ is nonsingular and
\begin{equation}\label{eq:right-scaled-gamma}
 \widehat\Gamma_T(S)=
 \begin{pmatrix}
 X_s&X_uD_u^{-T}\\
 (P_s-SX_s)D_s^T&P_u-SX_u
 \end{pmatrix}
\end{equation}
has a uniformly bounded inverse for $S\in\mathcal S$ and $T\ge T_0$.
For $0<\alpha_{\rm ter}<\gamma$, put
$v_T(t)=e^{-\alpha_{\rm ter}(T-t)}$.  The solution of
\begin{equation}\label{eq:right-green-problem}
 h_{t+1}=Mh_t+g_t,\qquad P_xh_0=0,\qquad R_Sh_T=b_T
\end{equation}
satisfies
\begin{equation}\label{eq:right-green-bound}
 \max_{0\le t\le T}\frac{\|h_t\|}{v_T(t)}
 \le C_{\to}\left(\|b_T\|+
 \max_{0\le s<T}\frac{\|g_s\|}{v_T(s+1)}\right),
\end{equation}
where $C_{\to}$ is independent of $(S,T)$.
\end{lemma}

\begin{proof}
In the dichotomy representation, the boundary system is
\eqref{eq:right-scaled-gamma}$(a,\eta)^\top=(b_L,b_R)^\top$, where, for
$G=\max_j\|g_j\|/v_T(j+1)$,
\[
 \begin{aligned}
 b_L&=P_xV_u\sum_{j<T}D_u^{-1-j}W_ug_j,\\
 b_R&=b_T-R_SV_s\sum_{j<T}D_s^{T-1-j}W_sg_j.
 \end{aligned}
\]
Geometric sums give
$b_L=O(v_T(0)G)$ and $b_R=O(\|b_T\|+G)$.  Uniform inversion bounds $\eta$;
the first row and $X_s^{-1}$ then give the extra estimate
$a=O(v_T(0)(\|b_T\|+G))$.  Substitution uses the summable ratios
$e^{-(\gamma+\alpha_{\rm ter})k}$ and
$e^{-(\gamma-\alpha_{\rm ter})k}$, proving the claim.
\end{proof}

The following quantitative row-perturbation result will be used for the
parameter-dependent terminal Hessian in Section~\ref{sec:branches}.

\begin{lemma}[Quantitative perturbation of endpoint rows]
\label{lem:boundary-row-perturbation}
Suppose that $(L,\Bzero,\BT)$ has the uniform Green property with constants
$(C_G,\gamma,T_0)$.  All induced operator norms below use the fixed
finite-dimensional norms of Definition~\ref{cond:green}.  For endpoint-row
perturbations $E_0,E_T$, put
\begin{equation}\label{eq:row-perturbation-size}
 \varepsilon=\|E_0\|+\|E_T\|,
 \qquad \vartheta=2C_G\varepsilon .
\end{equation}
If $\vartheta<1$, then $(L,\Bzero+E_0,\BT+E_T)$ has the same Green exponent
and initial horizon, with the common admissible constant
\begin{equation}\label{eq:perturbed-green-constant}
 C_G^{\rm pert}=C_G\frac{1+\vartheta}{1-\vartheta}.
\end{equation}
The conclusion is uniform whenever $\vartheta$ is bounded above by a fixed
number smaller than one.
\end{lemma}

\begin{proof}
Use the unperturbed inverse to write the perturbed problem as
\begin{equation}\label{eq:row-perturbation-fixed-point}
 h=\mathcal L_T^{\rm for}f+
   \mathcal L_T^{\rm bd}(b-E_0h_0-E_Th_T).
\end{equation}
Because
$\max_{0\le t\le T}\|(\mathcal L_T^{\rm bd}c)_t\|\le2C_G\|c\|$, its
right-hand side is a contraction with factor $\vartheta$, proving existence.
Set
\[
 \begin{aligned}
 m&=\max\{\|h_0\|,\|h_T\|\},\\
 F_\partial(f)&=\sum_{s=0}^{T-1}
 (e^{-\gamma s}+e^{-\gamma(T-s)})\|f_s\|.
 \end{aligned}
\]
At $t=0,T$, the boundary factor and the Green kernel are bounded by twice
the corresponding terms in $\|b\|+F_\partial(f)$.  Hence
\begin{equation}\label{eq:perturbed-endpoint-bound}
 m\le2C_G(\|b\|+F_\partial(f)+\varepsilon m)
 \le\frac{2C_G}{1-\vartheta}(\|b\|+F_\partial(f)).
\end{equation}
Writing $d_T(t)=e^{-\gamma t}+e^{-\gamma(T-t)}$, direct multiplication gives
\begin{equation}\label{eq:boundary-weight-product}
 d_T(t)(e^{-\gamma s}+e^{-\gamma(T-s)})
 \le2\kappa_T(t,s).
\end{equation}
Indeed, the cross-products are the two reflected kernels, whereas each of
the other products is bounded by $e^{-\gamma|t-s|}$.  Substitution of
\eqref{eq:perturbed-endpoint-bound} into the original pointwise Green
estimate yields
\[
 \begin{aligned}
 \|h_t\|&\le\frac{C_G}{1-\vartheta}d_T(t)\|b\|\\
 &\quad{}+\frac{C_G(1+\vartheta)}{1-\vartheta}
 \sum_{s=0}^{T-1}\kappa_T(t,s)\|f_s\|,
 \end{aligned}
\]
and hence Definition~\ref{cond:green} with
\eqref{eq:perturbed-green-constant}.  For homogeneous data,
\eqref{eq:perturbed-endpoint-bound} gives $h_0=h_T=0$, and unperturbed
well-posedness then gives $h=0$; thus the perturbed problem is well posed.
\end{proof}

\section{Nonlinear reconstruction in weighted spaces}\label{sec:reconstruction}

The nonlinear fixed-point argument uses endpoint weights matched to the Green kernel.  For $0<\alpha<\gamma/2$, set
\begin{equation}\label{eq:weight}
      w_T(t)=e^{-\alpha t}+e^{-\alpha(T-t)}.
\end{equation}
For each horizon define the weighted space
\[
\begin{aligned}
   X_T&=\{z=(z_0,\ldots,z_T):\|z\|_{w,T}<\infty\},\\
   \|z\|_{w,T}&=\sup_{0\le t\le T}\frac{\|z_t\|}{w_T(t)} .
\end{aligned}
\]
All differentiability statements below concern the family of maps into these horizon-dependent spaces uniformly in $T$.

\medskip\noindent\textit{Weighted convolution.}
If $0<\alpha<\gamma/2$, then there is a constant $C_{\rm conv}$, independent of $T$ and $t$, such that
\[
      \sum_{s=0}^{T-1}\kappa_T(t,s)w_T(s)^2
      \le C_{\rm conv}w_T(t)^2.
\]
Moreover $e^{-\gamma t}+e^{-\gamma(T-t)}\le C_{\rm conv}w_T(t)^2$, with $\kappa_T$ from \eqref{eq:kappa-kernel}.
For example, with $q_-=e^{-(\gamma-2\alpha)}$ and
$q_+=e^{-(\gamma+2\alpha)}$, one may use the explicit value
\begin{equation}\label{eq:explicit-convolution-constant}
 C_{\rm conv}=4\left(\frac1{1-q_-}+\frac{q_+}{1-q_+}\right).
\end{equation}
To verify this value, put $\chi_s^-=e^{-\alpha s}$,
$\chi_s^+=e^{-\alpha(T-s)}$, and define
\[
 A_{\gamma,\alpha}=\frac1{1-q_-}+\frac{q_+}{1-q_+}.
\]
Set
$\kappa_T^{\rm ref}(t,s)=e^{-\gamma(t+T-s)}+e^{-\gamma(T-t+s)}$.
Since $w_T(s)^2\le2((\chi_s^-)^2+(\chi_s^+)^2)$, splitting the interior
sum at $s=t$ gives the two geometric ratios $q_-$ and $q_+$ and hence
\begin{align}
 \sum_{s=0}^{T-1}e^{-\gamma|t-s|}w_T(s)^2
 &\le2A_{\gamma,\alpha}\bigl((\chi_t^-)^2+(\chi_t^+)^2\bigr),
 \label{eq:interior-weight-convolution}\\
 \sum_{s=0}^{T-1}\kappa_T^{\rm ref}(t,s)w_T(s)^2
 &\le2A_{\gamma,\alpha}\bigl((\chi_t^-)^2+(\chi_t^+)^2\bigr).
 \label{eq:reflected-weight-convolution}
\end{align}
For the reflected kernels, the coefficients of $(\chi_t^-)^2$ are
$q_+/(1-q_+)$ and $1/(1-q_-)$; those of $(\chi_t^+)^2$ are
$q_-/(1-q_-)$ and $1/(1-q_+)$.  Both pairs sum to
$A_{\gamma,\alpha}$.  Because
$(\chi_t^-)^2+(\chi_t^+)^2\le w_T(t)^2$, the two estimates prove
\eqref{eq:explicit-convolution-constant}.  Moreover, $\gamma>2\alpha$
implies
\[
 e^{-\gamma t}+e^{-\gamma(T-t)}
 \le(\chi_t^-)^2+(\chi_t^+)^2\le w_T(t)^2.
\]
Thus the same squared weight controls the interior and endpoint remainders.

\medskip\noindent\textit{Quadratic residual.}
Let $\mathcal F$ be $C^2$ near $(0,0)$, assume $\mathcal F(0,0)=0$, and let $L=D\mathcal F(0,0)$.  Define
\[
      N(a,b)=\mathcal F(a,b)-L(a,b).
\]
After shrinking the neighborhood, there is a constant $C_N$ such that
\[
      \norm{N(a,b)}\le C_N(\norm a+\norm b)^2
\]
and
\[
\begin{aligned}
      \norm{N(a,b)-N(\tilde a,\tilde b)}
      &\le C_N(\norm a+\norm b+\norm{\tilde a}+\norm{\tilde b})\\
      &\quad{}\times
          (\norm{a-\tilde a}+\norm{b-\tilde b}).
\end{aligned}
\]
Both estimates follow directly from Taylor's formula with integral remainder
on a fixed convex neighborhood.

We now combine the inverse estimate, the convolution bound, and the residual
estimate.  Affine endpoint rows are included as the case with zero remainder.
\begin{theorem}[Uniform reconstruction with smooth endpoint rows]
\label{thm:nonlinear-boundary}
Let $\mathcal F:\R^n\times\R^n\to\R^n$ be $C^2$ near $(0,0)$, satisfy
$\mathcal F(0,0)=0$, and have the quadratic residual bounds above.  Assume
that its linearized equation with rows
$(\Bzero,\BT)$ satisfies Definition~\ref{cond:green} with exponent $\gamma$.
Fix $0<\alpha<\gamma/2$.  Let $\B:\R^n\times\R^n\to\R^n$ be $C^2$ near the
origin, with bounded second derivative, and satisfy
$\B(0,0)=0$ and
$D\B(0,0)[a,b]=\Bzero a+\BT b$ (the affine row is allowed).  Put
$R_B(a,b)=\B(a,b)-\Bzero a-\BT b$ and choose $C_B$ such that
\begin{align}
 \|R_B(a,b)\|&\le C_B(\|a\|+\|b\|)^2,\label{eq:boundary-quadratic}\\
 \|R_B(a,b)-R_B(\tilde a,\tilde b)\|
 &\le C_B\Sigma\bigl(\|a-\tilde a\|+\|b-\tilde b\|\bigr),
 \label{eq:boundary-lipschitz}
\end{align}
where $\Sigma=\|a\|+\|b\|+\|\tilde a\|+\|\tilde b\|$.

Let $C_N$ be the interior residual constant above, set $c=C_G$, and let
\begin{equation}\label{eq:explicit-a-constant}
 a=2C_GC_{\rm conv}C_N(1+e^\alpha)^2+16C_GC_B .
\end{equation}
Choose any $R_*>0$ such that the weighted ball of radius $R_*$ lies in the
fixed smoothness tube and $2aR_*<1$, and define
\begin{equation}\label{eq:apriori-data-radius}
 \delta_*:=\frac{R_*-aR_*^2}{c}.
\end{equation}
Then, for $T\ge T_0$ and $d:=\|\Delta\|<\delta_*$, the problem
\begin{equation}\label{eq:nonlinear-boundary}
 \mathcal F(z_t,z_{t+1})=0,\qquad \B(z_0,z_T)=\Delta
\end{equation}
has a solution in the explicitly computable ball
\begin{equation}\label{eq:apriori-root-radius}
 \|z\|_{w,T}\le R_-(d):=
 \frac{1-\sqrt{1-4acd}}{2a}<R_* ,
\end{equation}
where, if $a=0$, the continuous convention is $R_-(d)=cd$.  It is the only
solution in $\|z\|_{w,T}\le R_*$.  With
$z^{\lin}(\Delta):=\mathcal L_T^{\bd}\Delta$ and
$E_t(\Delta):=z_t(\Delta)-z_t^{\lin}(\Delta)$, uniformly in $T$,
\begin{align}
\|z_t(\Delta)\|&\le Cw_T(t)\|\Delta\|+Cw_T(t)^2\|\Delta\|^2,
 \label{eq:nonlinear-boundary-estimate}\\
\|E_t(\Delta)\|&\le Cw_T(t)^2\|\Delta\|^2.
 \label{eq:affine-reconstruction-estimate}\\
\|z_t(\Delta_1)-z_t(\Delta_2)\|&\le
Cw_T(t)\|\Delta_1-\Delta_2\|.
\label{eq:nonlinear-boundary-lipschitz}
\end{align}
Moreover,
\begin{equation}
\label{eq:nonlinear-boundary-linearization}
\begin{aligned}
\|E_t(\Delta_1)-E_t(\Delta_2)\|
&\le Cw_T(t)^2(\|\Delta_1\|+\|\Delta_2\|)\\
&\quad{}\times \|\Delta_1-\Delta_2\|.
\end{aligned}
\end{equation}
Consequently the solution family is Fr\'echet differentiable at $0$ uniformly
in $T$, with derivative $\mathcal L_T^{\bd}$.
\end{theorem}

\begin{proof}
Let $\mathcal A_T$ denote the linear interior and boundary operator, and let
$\mathcal N_T(z)$ collect $N(z_t,z_{t+1})$ and $R_B(z_0,z_T)$.  The Green
estimate, $w_T(t+1)\le e^\alpha w_T(t)$, and the weighted convolution give
\begin{align}
 \|\mathcal A_T^{-1}\mathcal N_T(z)\|_{w,T}
 &\le a\|z\|_{w,T}^2,\label{eq:preconditioned-quadratic}\\
 \|\mathcal A_T^{-1}(\mathcal N_T(z)-\mathcal N_T(\tilde z))\|_{w,T}
 &\le a(\|z\|_{w,T}+\|\tilde z\|_{w,T})\nonumber\\
 &\quad\times\|z-\tilde z\|_{w,T}.
 \label{eq:preconditioned-lipschitz}
\end{align}
Indeed, the interior contribution is bounded by the first term in
\eqref{eq:explicit-a-constant}.  At the endpoints
$w_T(0),w_T(T)\le2$, so the two factors in
\eqref{eq:boundary-lipschitz} are at most
$4(\|z\|_{w,T}+\|\tilde z\|_{w,T})$ and
$4\|z-\tilde z\|_{w,T}$; this gives the second
term in \eqref{eq:explicit-a-constant}.

Define
\[
 \Phi_T(z)=\mathcal L_T^{\rm bd}\Delta-
            \mathcal A_T^{-1}\mathcal N_T(z).
\]
Since $\|\mathcal L_T^{\rm bd}\Delta\|_{w,T}\le cd$,
\eqref{eq:preconditioned-quadratic} shows that the ball of radius $R_-(d)$
is invariant: $cd+aR_-^2=R_-$.  Equation
\eqref{eq:preconditioned-lipschitz} gives contraction factor
$2aR_-<1$.  Banach's theorem proves existence.  If $z,\tilde z$ are any two
solutions in the ball of radius $R_*$, then
\[
 \|z-\tilde z\|_{w,T}\le2aR_*\|z-\tilde z\|_{w,T},
\]
so they coincide.  This proves the computable radius and uniqueness claims.

Subtracting $z^{\lin}$ from the fixed-point equation leaves the quadratic
interior and endpoint residuals.  The sharper convolution estimate for $w_T^2$ and
$e^{-\gamma t}+e^{-\gamma(T-t)}\le Cw_T(t)^2$ give
\eqref{eq:affine-reconstruction-estimate} and then
\eqref{eq:nonlinear-boundary-estimate}.  For two data $\Delta_1,\Delta_2$
with $d_i:=\|\Delta_i\|<\delta_*$, write
\[
 R_i=R_-(d_i),\qquad
 q_{12}=a(R_1+R_2)<1.
\]
Subtracting the fixed-point equations and applying
\eqref{eq:preconditioned-lipschitz} gives
\begin{equation}\label{eq:explicit-two-data-lipschitz}
 \|z_1-z_2\|_{w,T}
 \le\frac{c}{1-q_{12}}\|\Delta_1-\Delta_2\|.
\end{equation}
For the pointwise remainder difference, set
\begin{equation}\label{eq:pointwise-residual-constant}
 \bar a=C_GC_{\rm conv}C_N(1+e^\alpha)^2+16C_GC_B\le a.
\end{equation}
The interior residual difference is at most
\[
 C_N(1+e^\alpha)^2w_T(s)^2(R_1+R_2)
 \|z_1-z_2\|_{w,T},
\]
and \eqref{eq:boundary-lipschitz}, together with
$w_T(0),w_T(T)\le2$, gives the boundary bound
$16C_B(R_1+R_2)\|z_1-z_2\|_{w,T}$.  The estimates for $w_T^2$ therefore
yield
\begin{equation}\label{eq:explicit-two-data-remainder}
 \|E_t(\Delta_1)-E_t(\Delta_2)\|
 \le\frac{\bar a c(R_1+R_2)}{1-q_{12}}
 w_T(t)^2\|\Delta_1-\Delta_2\|.
\end{equation}
Since $R_i=cd_i+aR_i^2$ and $aR_i<aR_*<1/2$,
$R_i\le2cd_i$ and $q_{12}\le2aR_*$.  Consequently
\eqref{eq:nonlinear-boundary-lipschitz} and
\eqref{eq:nonlinear-boundary-linearization} hold with
\begin{equation}\label{eq:explicit-two-data-constants}
 C_{\rm Lip}=\frac{c}{1-2aR_*},\qquad
 C_{\rm rem}=\frac{2\bar a c^2}{1-2aR_*}.
\end{equation}
All constants depend only on the displayed uniform bounds, not on $T$.
\end{proof}

The restriction $\alpha<\gamma/2$ is needed for the remainder with weight
$w_T^2$.  Existence, size, and Lipschitz continuity alone hold for every
$0<\alpha<\gamma$ by the corresponding convolution estimate for $w_T$.

\section{Application to Pontryagin boundary conditions}\label{sec:branches}

We now apply Theorem~\ref{thm:nonlinear-boundary} directly to the reduced
Pontryagin equation and its nonlinear terminal condition.

\begin{theorem}[Uniform local Pontryagin branch with a nonlinear terminal condition]\label{thm:main}
Assume the local hypotheses of Section~\ref{sec:model}, including
\eqref{eq:regular-stationarity}, and let
$y_\infty\in Y$ satisfy $\pinf=r_x(\xinf,y_\infty)$.  For $y$ near
$y_\infty$ define
\[
\begin{aligned}
   \B_y(z_0,z_T)&=
   \binom{\xi_0}
         {\pi_T-\bigl(r_x(\xinf+\xi_T,y)-r_x(\xinf,y)\bigr)},\\
   \Delta_{\rm oc}(x_{\rm in},y)&=
   \binom{x_{\rm in}-\xinf}{r_x(\xinf,y)-\pinf},
\end{aligned}
\]
where $z=(\xi,\pi)=(x-\xinf,p-\pinf)\in\R^{2d}$.  Then
$\B_y(z_0,z_T)=\Delta_{\rm oc}(x_{\rm in},y)$ is equivalent to
\[
        x_0=x_{\rm in},\qquad p_T=r_x(x_T,y).
\]
If the rows $D\B_{y_\infty}(0,0)$ satisfy Definition~\ref{cond:green} with
constants $(C_G,\gamma,T_0)$, choose a relatively compact neighborhood
$Y_1\Subset Y_0$.  With
\[
 E_T(y)=\begin{pmatrix}
 0&0\\-[r_{xx}(\xinf,y)-r_{xx}(\xinf,y_\infty)]&0
 \end{pmatrix},
\]
shrink $Y_1$ so that
\begin{equation}\label{eq:y-row-margin}
 \bar\vartheta:=2C_G\sup_{y\in Y_1}
 \|E_T(y)\|<1.
\end{equation}
Then Lemma~\ref{lem:boundary-row-perturbation} gives the common
Green constant
\begin{equation}\label{eq:y-common-green}
 \bar C_G=C_G\frac{1+\bar\vartheta}{1-\bar\vartheta}.
\end{equation}
Set $C_B=M_{\rm ter}$ in
\eqref{eq:boundary-quadratic}--\eqref{eq:boundary-lipschitz}, define the
common constants $a,c$ from \eqref{eq:explicit-a-constant} using
$\bar C_G$, and choose one pair $R_*,\delta_*$ as in
\eqref{eq:apriori-data-radius}.  Then, for every $T\ge T_0$ and every
$(x_{\rm in},y)$ satisfying
$d:=\|\Delta_{\rm oc}(x_{\rm in},y)\|<\delta_*$, the problem admits a unique local
stationary Pontryagin branch in the fixed ball $\|z\|_{w,T}\le R_*$.  Its
radius is bounded by \eqref{eq:apriori-root-radius}, uniformly in $T$ and
$y$.

For fixed $y$, let $z^{\lin}_{T,y}=(\xi^{\lin},\pi^{\lin})$ be the linear
branch formed with $D\B_y(0,0)$, and put
\[
 u_t^{\lin}=D\nu(\xinf,\pinf)
       [\xi_t^{\lin},\pi_{t+1}^{\lin}].
\]
Uniformly for $0\le t<T$,
\begin{align}
 \|z_t\|+\|u_t-\uinf\|&\le Cw_T(t)d,\label{eq:main-size}\\
 \|z_t-z^{\lin}_{T,y,t}\|+
 \|(u_t-\uinf)-u_t^{\lin}\|&\le Cw_T(t)^2d^2.
 \label{eq:main-first}
\end{align}
For two branches with the same $y$,
\begin{equation}\label{eq:main-lip}
 \|z_t^1-z_t^2\|+\|u_t^1-u_t^2\|
 \le Cw_T(t)\|\Delta_{\rm oc}^1-\Delta_{\rm oc}^2\|.
\end{equation}
The $z$-estimates also hold at $t=T$.  The expansion is with respect to the
combined boundary datum $\Delta_{\rm oc}$; at fixed $y$ it is an expansion at
$x_{\rm in}=\xinf$ only when $r_x(\xinf,y)=\pinf$.

If, in addition,
$y\mapsto r_x(x,y)$ and $y\mapsto r_{xx}(x,y)$ are locally Lipschitz uniformly
in $x$ on the terminal tube, then the branch is jointly locally Lipschitz in
$(x_{\rm in},y)$ on this neighborhood of the endpoint data.

The Green hypothesis is verified directly from problem data if
\eqref{eq:explicit-transition} holds, $M$ is hyperbolic, and the matrices
$\widehat\Gamma_T$ for $D\B_{y_\infty}(0,0)$ satisfy the finite-horizon
certificate of Theorem~\ref{prop:scaled-green}.
\end{theorem}

\begin{proof}
The boundary identity is immediate from the definitions.  Its derivative is
\[
 D\B_y(0,0)[z_0,z_T]
 =\binom{\xi_0}{\pi_T-r_{xx}(\xinf,y)\xi_T}.
\]
The terminal remainder is
\[
 -\binom0{\displaystyle
 \int_0^1[r_{xx}(\xinf+\theta\xi_T,y)-r_{xx}(\xinf,y)]
 \xi_T\,d\theta},
\]
so the uniform $D_x^3r$ bound gives
\eqref{eq:boundary-quadratic}--\eqref{eq:boundary-lipschitz} with
$C_B=M_{\rm ter}$.  Continuity of $r_{xx}$ in $y$ permits
\eqref{eq:y-row-margin}.  The only endpoint-row perturbation is $E_T(y)$; in
the standard Euclidean product norm, its norm is the norm of the displayed
difference between terminal Hessians.  Hence
Lemma~\ref{lem:boundary-row-perturbation} gives
\eqref{eq:y-common-green} uniformly on $Y_1$.

Theorem~\ref{thm:nonlinear-boundary} now gives the explicit radius, uniqueness,
and the estimates for $z$.  Set
$u_t=\nu(\xinf+\xi_t,\pinf+\pi_{t+1})$.  The fixed tube keeps these controls
in $\operatorname{int}U$, and \eqref{eq:legendre-graph} recovers the full
stationary Pontryagin equations.  Taylor's formula for $\nu$, together with
$w_T(t+1)\le e^\alpha w_T(t)$, transfers the size, remainder, and Lipschitz
estimates to $u$.

For two parameters $y_1,y_2$, subtract the interior and boundary equations.
The additional terms contain
$r_x(\cdot,y_1)-r_x(\cdot,y_2)$ and
$r_{xx}(\cdot,y_1)-r_{xx}(\cdot,y_2)$.  On an endpoint ball of radius
$\varepsilon$, the uniform Lipschitz bounds and the common Green estimate give
\begin{equation}\label{eq:joint-y-absorption}
 \begin{aligned}
 \|z^1-z^2\|_{w,T}
 &\le C\|x_{\rm in}^1-x_{\rm in}^2\|+C\|y_1-y_2\|\\
 &\quad{}+C\varepsilon\|z^1-z^2\|_{w,T}.
 \end{aligned}
\end{equation}
The constant is horizon independent by \eqref{eq:y-common-green}.  Shrinking
the fixed neighborhood of the endpoint data until $C\varepsilon<1/2$ absorbs the
last term and proves the joint assertion.  The final statement in terms of the problem data is
Theorem~\ref{prop:scaled-green} applied to
$M=-L_1^{-1}L_0$.
\end{proof}

\begin{proposition}[A posteriori existence and local uniqueness certificate]\label{prop:posterior}
Let $\mathscr R_{T,\Delta}$ be the full interior and boundary residual of
\eqref{eq:nonlinear-boundary}, obtained by stacking the $T$ interior
residuals and the boundary residual in $\R^{n(T+1)}$.  Equip the trajectory
space $X_T$ with a chosen norm $\|\cdot\|_X$; this may be
$\|\cdot\|_{w,T}$, while Section~\ref{sec:numerics} uses the ordinary vector
infinity norm.  Let $\bar z$ be an approximate solution and put
$J=D\mathscr R_{T,\Delta}(\bar z)$.  Suppose $J$ is invertible and define
\begin{equation}\label{eq:posterior-beta}
 \beta=\|J^{-1}\mathscr R_{T,\Delta}(\bar z)\|_X.
\end{equation}
Assume that, whenever $z$ and $\tilde z$ are in
$\|\cdot-\bar z\|_X\le R_{\rm tube}$,
\begin{equation}\label{eq:posterior-omega}
\begin{aligned}
 &\|J^{-1}(D\mathscr R_{T,\Delta}(z)-
 D\mathscr R_{T,\Delta}(\tilde z))\|_{\mathcal L(X_T)}\\
 &\hspace{5em}\le\omega\|z-\tilde z\|_X.
\end{aligned}
\end{equation}
If $2\omega\beta<1$ and
\begin{equation}\label{eq:posterior-radius}
 r_-:=\frac{1-\sqrt{1-2\omega\beta}}{\omega}<R_{\rm tube},
\end{equation}
then an exact solution exists in
$\|z-\bar z\|_X\le r_-$.  For every
\[
 r_-<R<\min\{R_{\rm tube},1/\omega\},
\]
it is the unique zero in $\|z-\bar z\|_X<R$.  For $\omega=0$, use the
conventions $r_-=\beta$ and
$1/\omega=+\infty$.
\end{proposition}

\begin{proof}
The frozen-Newton map
$\Psi(z)=z-J^{-1}\mathscr R_{T,\Delta}(z)$ satisfies
\begin{align*}
 \|\Psi(\bar z+h)-\bar z\|_X
 &\le\beta+\tfrac12\omega\|h\|_X^2,\\
 \|D\Psi(\bar z+h)\|_{\mathcal L(X_T)}&\le\omega\|h\|_X .
\end{align*}
Equation \eqref{eq:posterior-radius} is the smaller root of
$\beta+\omega r^2/2=r$ and obeys $\omega r_-<1$.  Banach's theorem proves
existence and uniqueness in the smaller ball.  On any larger ball with
$\omega r<1$, two zeros would be two fixed points of a strict contraction;
hence the computed zero is the only one there.
\end{proof}

All quantities in Proposition~\ref{prop:posterior} are finite-dimensional and
computable.  Double-precision evaluation is a diagnostic; a formal
computer-assisted certificate requires outward-rounded interval bounds for
$\beta$ and $\omega$.

\begin{corollary}[Uniform sensitivity of the initial control]\label{cor:feedback}
Under the hypotheses of Theorem~\ref{thm:main}, define
\[
   \Kfb_T(x_{\rm in},y):=u_0(x_{\rm in},y),
\]
the first control of the local stationary branch.  After possibly shrinking
the neighborhood of the endpoint data, there is a constant $L$ independent of
$(T,y)$ such that, for every fixed admissible $y$, every $T\ge T_0$, and every
two admissible initial states $x_1,x_2$,
\[
   \|\Kfb_T(x_1,y)-\Kfb_T(x_2,y)\|
   \le L\|x_1-x_2\|.
\]
If the final Lipschitz assumption of Theorem~\ref{thm:main} holds, then
for arbitrary admissible $(x_1,y_1)$ and $(x_2,y_2)$,
\[
\begin{aligned}
 \|\Kfb_T(x_1,y_1)-\Kfb_T(x_2,y_2)\|
 &\le L\|x_1-x_2\|\\
 &\quad{}+L\|y_1-y_2\|.
\end{aligned}
\]
Moreover, for each fixed $y$, set
$\Delta_{\rm oc}=\Delta_{\rm oc}(x_{\rm in},y)$.  With
\[
 z^{\lin}_{T,y}(\Delta_{\rm oc})
 =(\xi^{\lin},\pi^{\lin})
\]
denoting the linearized branch formed with $D\B_y(0,0)$,
\[
\begin{aligned}
   \Kfb_T(x_{\rm in},y)
   &=\uinf+
     D\nu(\xinf,\pinf)
     [\xi_0^{\lin},\pi_1^{\lin}]\\
   &\quad{}+O(\|\Delta_{\rm oc}(x_{\rm in},y)\|^2),
\end{aligned}
\]
with a constant independent of $(T,y)$.  This expansion treats the initial and
terminal boundary data jointly, subject to the qualification stated after
\eqref{eq:main-lip}.
\end{corollary}
\begin{proof}
Take $t=0$ in \eqref{eq:main-lip} and \eqref{eq:main-first}, and use
$w_T(0)\le2$.  The joint estimate follows from the last part of
Theorem~\ref{thm:main}.
\end{proof}

\begin{proposition}[Stationary objective and gradient with respect to the initial state]
\label{prop:critical-value}
For the branch of Theorem~\ref{thm:main}, define the objective evaluated along
the stationary branch by
\begin{equation}\label{eq:stationary-action}
 \mathscr S_T(x,y)=r(x_T(x,y),y)
 -\lambda\sum_{t=0}^{T-1}\ell(x_t(x,y),u_t(x,y)).
\end{equation}
After shrinking the data neighborhood, the branch is $C^1$ in $x$ for each
fixed admissible $y$, and
\begin{equation}\label{eq:critical-value-gradient}
 D_x\mathscr S_T(x,y)=p_0(x,y),\qquad
 \operatorname{Lip}(D_x\mathscr S_T)\le C,
\end{equation}
with $C$ independent of $T$ for fixed admissible $y$.
If the stationary branch is the unique maximizer in the specified tube, then
$\mathscr S_T$ is the corresponding local value function and $\Kfb_T$ is its
initial optimal feedback map.  Otherwise, they describe only the objective
value and initial control associated with a stationary branch.
\end{proposition}

\begin{proof}
Finite-dimensional implicit differentiation applies for each $T$; the common
inverse bounds give the stated uniform estimates.  For a variation of $x_0$,
stationarity, the costate equation, and $p_T=r_x$ give
\[
\begin{aligned}
D_x\mathscr S_T h
&=p_T^\top\delta x_T-\lambda\sum_t
 (\ell_x^\top\delta x_t+\ell_u^\top\delta u_t)\\
&=p_T^\top\delta x_T+\sum_t
 (p_t^\top\delta x_t-p_{t+1}^\top\delta x_{t+1})=p_0^\top h.
\end{aligned}
\]
Equation \eqref{eq:main-lip} at $t=0$ proves the $C^{1,1}$ bound.  At the
reference point $(x,y)=(\xinf,y_\infty)$, \eqref{eq:main-first} also yields a horizon-uniform quadratic
expansion of $p_0$, and hence a cubic remainder for $\mathscr S_T$.
\end{proof}

\begin{theorem}[Exponential decay of sensitivity to the terminal reward]
\label{thm:terminal-reward}
Assume the explicit hyperbolic transition \eqref{eq:explicit-transition} and
the graph certificate of Lemma~\ref{lem:right-green} at a reference terminal
Hessian $S_\star$.  Let $\mathcal X$ be a convex terminal-state neighborhood
of $\xinf$ and let
$\mathscr R$ be a class of terminal rewards with a uniform $C^3$ bound.  Fix
$0<\alpha_{\rm ter}<\gamma$ and assume
\begin{equation}\label{eq:terminal-reward-class}
 \sup_{\rho\in\mathscr R}\sup_{x\in\mathcal X}
 \|D^2\rho(x)-S_\star\|\le\eta .
\end{equation}
Put
\[
 \begin{aligned}
 G_\Gamma&=\sup_{T\ge T_0}\|\widehat\Gamma_T(S_\star)^{-1}\|,\\
 C_X&=\sup_{T\ge T_0}\|[X_sD_s^{\,T}\;X_u]\|,
 \qquad G_\Gamma C_X\eta<1.
 \end{aligned}
\]
Let $C_{\to}$ be the resulting common constant in
Lemma~\ref{lem:right-green}, and let $C_N$ be a common residual constant in
Theorem~\ref{thm:nonlinear-boundary}.  Assume that common branch and data radii
$R$ and $\delta_*$ have been chosen so that
\begin{equation}\label{eq:terminal-absorption}
 \theta:=8C_{\to}\|L_1^{-1}\|C_N(1+e^{-\alpha_{\rm ter}})R<1 .
\end{equation}
Assume also that the initial-state neighborhood and $\mathscr R$ satisfy
$d_\rho(x)=\|(x-\xinf,D\rho(\xinf)-\pinf)\|<\delta_*$ and the resulting
terminal states lie in $\mathcal X$.  Then
Theorem~\ref{thm:nonlinear-boundary}, applied to each endpoint map, gives local
branches on the same neighborhood.  For the same admissible initial state,
every $\rho,\widetilde\rho\in
\mathscr R$, $T\ge T_0$, and $0\le t<T$,
\begin{align}
 &\|z_t^\rho-z_t^{\widetilde\rho}\|
 +\|u_t^\rho-u_t^{\widetilde\rho}\|\nonumber\\
 &\qquad\le C_{\rm ter}e^{-\alpha_{\rm ter}(T-t)}
 \|D(\rho-\widetilde\rho)\|_{L^\infty(\mathcal X)} .
 \label{eq:terminal-trajectory-attenuation}
\end{align}
In particular,
\begin{equation}\label{eq:terminal-policy-attenuation}
 \boxed{\ \|\Kfb_T^\rho(x)-\Kfb_T^{\widetilde\rho}(x)\|
 \le C_{\rm ter}e^{-\alpha_{\rm ter}T}
 \|\rho-\widetilde\rho\|_{C^2(\mathcal X)}\ } .
\end{equation}
The gradients with respect to the initial state in
\eqref{eq:critical-value-gradient} obey the same
$e^{-\alpha_{\rm ter}T}$ bound.  The constants are horizon independent
for every fixed $0<\alpha_{\rm ter}<\gamma$.
\end{theorem}

\begin{proof}
The perturbation of \eqref{eq:right-scaled-gamma} caused by $S-S_\star$ is
confined to its lower row and has norm at most $C_X\eta$.  Neumann's lemma
therefore gives the uniform inverse bound
$G_\Gamma/(1-G_\Gamma C_X\eta)$ and hence one common $C_{\to}$.

Set $h=z^\rho-z^{\widetilde\rho}$ and
\[
 \Sigma_T=\int_0^1D^2\rho(x_T^{\widetilde\rho}
 +s(x_T^\rho-x_T^{\widetilde\rho}))\,ds .
\]
Subtracting the boundary conditions gives
\begin{equation}\label{eq:terminal-difference-row}
 P_xh_0=0,\qquad R_{\Sigma_T}h_T
 =D(\rho-\widetilde\rho)(x_T^{\widetilde\rho})=:b_T.
\end{equation}
Writing $\widetilde{\mathcal F}=L+N$, the interior difference becomes
$h_{t+1}=Mh_t+g_t$, where
\[
 g_t=-L_1^{-1}\{N(z_t^\rho,z_{t+1}^\rho)
 -N(z_t^{\widetilde\rho},z_{t+1}^{\widetilde\rho})\}.
\]
If $H=\max_t\|h_t\|/v_T(t)$, the residual Lipschitz bound, $w_T\le2$, and
$v_T(t)=e^{-\alpha_{\rm ter}}v_T(t+1)$ imply
\[
 \max_t\frac{\|g_t\|}{v_T(t+1)}
 \le8\|L_1^{-1}\|C_N(1+e^{-\alpha_{\rm ter}})RH.
\]
Lemma~\ref{lem:right-green} and \eqref{eq:terminal-absorption} now give
$H\le C_{\to}\|D(\rho-\widetilde\rho)\|_\infty+\theta H$.
Thus $C_z=C_{\to}/(1-\theta)$ controls the state--costate difference, and
$C_{\rm ter}=C_z[1+L_\nu(1+e^{\alpha_{\rm ter}})]$ is admissible, where
$L_\nu=\sup\|D\nu\|$.  Indeed, the control estimate follows from
$u_t=\nu(x_t,p_{t+1})$; at $t=0$, the common initial state and
$v_T(1)=e^{\alpha_{\rm ter}}e^{-\alpha_{\rm ter}T}$ give
\eqref{eq:terminal-policy-attenuation}.  The gradient estimate follows from
\eqref{eq:critical-value-gradient}.
\end{proof}

\section{Verification for definite linear-quadratic systems}\label{sec:lq}
After absorbing the positive multiplier $\lambda$ into $Q$ and $R$, consider
the LQ dynamics $x_{t+1}=Ax_t+Bu_t$ with running term
$-\frac12x_t^\top Qx_t-\frac12u_t^\top Ru_t$, $Q\succeq0$, $R\succ0$.
The Hamiltonian stationarity equation is $u_t=R^{-1}B^\top p_{t+1}$ and the
state--costate equations are
\begin{equation}\label{eq:lq-system-ric}
\begin{aligned}
      x_{t+1}&=Ax_t+BR^{-1}B^\top p_{t+1},\\
      p_t&=A^\top p_{t+1}-Qx_t.
\end{aligned}
\end{equation}
Assume $A$ is invertible, put $G=BR^{-1}B^\top$, and eliminate
$p_{t+1}$ to obtain
{\setlength{\arraycolsep}{3pt}
\begin{equation}\label{eq:lq-M}
      M(A,B,Q,R)=
      \begin{pmatrix}
        A+GA^{-\top}Q & GA^{-\top}\\
        A^{-\top}Q & A^{-\top}
      \end{pmatrix}.
\end{equation}
}
Then $\binom{x_{t+1}}{p_{t+1}}=M(A,B,Q,R)\binom{x_t}{p_t}$, and let $\Jc=\bigl(\begin{smallmatrix}0&I\\-I&0\end{smallmatrix}\bigr)$.
The factorization
\[
 M=
 \begin{pmatrix}I&G\\0&I\end{pmatrix}
 \begin{pmatrix}A&0\\0&A^{-\top}\end{pmatrix}
 \begin{pmatrix}I&0\\Q&I\end{pmatrix}
\]
consists of symplectic factors, hence $M^\top\Jc M=\Jc$ and the spectrum is
reciprocal \citep{bittanti1991riccati}.

\begin{proposition}[Symplectic graph transversality]
\label{prop:symplectic-graph}
Let $M\in\mathrm{Sp}(2d,\R)$ be hyperbolic.  Write bases of its stable and
unstable subspaces as
\[
 \begin{aligned}
 V_s&=\binom{X_s}{P_s},\qquad V_u=\binom{X_u}{P_u},\\
 MV_s&=V_sD_s,\qquad MV_u=V_uD_u.
 \end{aligned}
\]
For $S=S^\top$, consider the endpoint rows
\begin{equation}\label{eq:graph-endpoint-rows}
 \Bzero=\begin{pmatrix}I&0\\0&0\end{pmatrix},\qquad
 \BT=\begin{pmatrix}0&0\\-S&I\end{pmatrix},
\end{equation}
which impose $x_0=\xi$ and $p_T-Sx_T=\pi$.  The subspaces $E_s,E_u$ are
Lagrangian, and
\begin{equation}\label{eq:lq-limit-matrix}
 \widehat\Gamma_\infty=[\Bzero V_s\quad\BT V_u]
 =\begin{pmatrix}X_s&0\\0&P_u-SX_u\end{pmatrix}.
\end{equation}
The two blocks are nonsingular exactly when
\begin{equation}\label{eq:graph-intersections}
 \begin{aligned}
 E_s\cap(\{0\}\times\R^d)&=\{0\},\\
 E_u\cap\{(x,Sx):x\in\R^d\}&=\{0\}.
 \end{aligned}
\end{equation}
Under these graph conditions, the uniform Green property holds for all
sufficiently large horizons.  For a prescribed $T_0$, it holds for every
$T\ge T_0$ if the finitely many scaled matrices $\widehat\Gamma_T$ below a
large-horizon cutoff are also nonsingular.  The conclusions are open under
small perturbations of $(M,S)$ within the hyperbolic symplectic region.  The
large-horizon conclusion is uniform on compact subsets satisfying the two
graph conditions.  Uniformity from a prescribed $T_0$ additionally requires
the finitely many short-horizon matrices to remain nonsingular on the compact
set.
\end{proposition}

\begin{proof}
For $v,w\in E_s$, symplecticity gives
\[
 v^\top\Jc w=(M^kv)^\top\Jc(M^kw)\longrightarrow0
 \quad (k\to\infty).
\]
Thus $E_s$ is isotropic.  Reciprocal spectral symmetry gives
$\dim E_s=\dim E_u=d$, so $E_s$ is Lagrangian; applying the same argument to
$M^{-k}$ on $E_u$ proves the unstable assertion.  Direct multiplication by
\eqref{eq:graph-endpoint-rows} gives \eqref{eq:lq-limit-matrix}, and the two
kernel conditions are precisely \eqref{eq:graph-intersections}.

Moreover,
$\|\widehat\Gamma_T-\widehat\Gamma_\infty\|\le Ce^{-\gamma T}$.
If the diagonal blocks in \eqref{eq:lq-limit-matrix} are nonsingular,
Neumann's lemma gives a uniform inverse for all sufficiently large $T$;
Theorem~\ref{prop:scaled-green} then gives the Green property.  The finitely
many shorter horizons are covered by their assumed nonsingularity.
Continuity of hyperbolic spectral projectors and the two graph blocks proves
openness and uniformity on compact subsets for the large-horizon conclusion.  For fixed
$T_0$, continuity of the finitely many remaining singular values gives the
corresponding assertion for every $T\ge T_0$.
\end{proof}

Recall that $(A,B)$ is stabilizable if
$\operatorname{rank}[A-\lambda I\ \ B]=d$ for every $|\lambda|\ge1$.
Along every complex solution of \eqref{eq:lq-system-ric}, direct substitution
gives the energy identity
\begin{equation}\label{eq:energy}
\begin{aligned}
      p_{t+1}^*x_{t+1}-p_t^*x_t
      &=p_{t+1}^*Gp_{t+1}+x_t^*Qx_t,\\
      G&=BR^{-1}B^\top .
\end{aligned}
\end{equation}

\begin{theorem}[Stabilizable definite LQ data]\label{thm:lq}
Let $A$ be invertible, $Q=Q^\top\succ0$, $R=R^\top\succ0$, and let
$(A,B)$ be stabilizable.  Then $M(A,B,Q,R)$ is hyperbolic, $X_s$ is
invertible, and for every $S=S^\top\preceq0$ the block $P_u-SX_u$ is
invertible.  The rows
\[
       x_0=\xi,\qquad p_T-Sx_T=\pi
\]
satisfy the uniform Green property for every $T\ge1$.  The constants can be
chosen uniformly on compact subsets of this data class.
\end{theorem}

\begin{proof}
If $Mz=\lambda z$ with $|\lambda|=1$, the left side of
\eqref{eq:energy} vanishes along $\lambda^tz$.  Hence $x=0$ and
$B^\top p=0$.  The second block equation gives
$A^\top p=\lambda^{-1}p$, contradicting the PBH condition unless $p=0$.
Thus $M$ is hyperbolic.

If $(0,p)\in E_s$, sum \eqref{eq:energy} forward to infinity.  Exponential
decay makes both endpoint pairings vanish, hence $x_t=0$ and
$B^\top p_t=0$.  Here $p_{t+1}=A^{-\top}p_t$.  If $p\ne0$, then
$V=\operatorname{span}\{p_t:t\ge1\}$ is nonzero and invariant under
$A^{-\top}$.  Since $A^{-\top}$ is injective, $A^{-\top}V=V$, so $V$ is also
$A^\top$-invariant.  The powers of $A^{-\top}|_V$ converge to zero, hence
$A^\top|_V$ has an eigenvalue of modulus greater than one.  Its eigenvector is
annihilated by $B^\top$, contradicting the PBH condition.  Therefore $X_s$
is invertible.

Let $(x_0,Sx_0)\in E_u$ with $S\preceq0$.  Summing the energy identity from
$-\infty$ to $-1$ gives a nonnegative sum equal to
$x_0^\top Sx_0\le0$.  Thus $x_t=0$ for $t\le-1$ and $Gp_0=0$; the state
equation gives $x_0=0$ and then $p_0=0$.  Hence
$P_u-SX_u$ is invertible.  By \eqref{eq:lq-limit-matrix} and
Theorem~\ref{prop:scaled-green}, the Green property holds for all sufficiently
large horizons.

For a finite homogeneous problem, $x_0=0$ and $p_T=Sx_T$.  Summing
\eqref{eq:energy} gives
\[
 \sum_{t=0}^{T-1}
 \bigl(p_{t+1}^*Gp_{t+1}+x_t^*Qx_t\bigr)
 =x_T^*Sx_T\le0.
\]
The sum is nonnegative, so both sides vanish.  Consequently $x_t=0$,
$Gp_T=0$, then $x_T=0$, $p_T=0$, and backward
recursion gives $p_t=0$.  Thus every finite-horizon boundary matrix is
invertible.  The finitely many matrices below the large-horizon cutoff are
absorbed into the uniform bound.  Continuous dependence of invariant
subspaces and finite-horizon matrices gives compact-subset uniformity.
\end{proof}

\begin{corollary}[Riccati representation and sensitivity to the terminal Hessian]
\label{cor:riccati-rate}
Under Theorem~\ref{thm:lq}, let the terminal reward be
$r_S(x)=\tfrac12x^\top Sx$ with $S\preceq0$.  For the unconstrained problem
(or where its optimizer is interior), strict concavity makes the stationary
branch the unique optimal solution of the Bellman problem, with
\[
 V_T^S(x)=\tfrac12x^\top\mathscr P_T(S)x,\qquad
 \Kfb_T^S(x)=K_T(S)x.
\]
Fix $K\ge1$ and $\gamma>0$ such that, for every $k\ge0$,
$\|V_sD_s^k\|\le Ke^{-\gamma k}$ and
$\|V_uD_u^{-k}\|\le Ke^{-\gamma k}$.
Put
\begin{equation}\label{eq:lq-mobius-data}
 \begin{aligned}
 C_s&=P_s-SX_s,\qquad C_u=P_u-SX_u,\\
 E_T^{\rm Ric}&=D_u^{-T}C_u^{-1}C_sD_s^T .
 \end{aligned}
\end{equation}
Then
\begin{align}
 \mathscr P_T(S)&=(P_s-P_uE_T^{\rm Ric})(X_s-X_uE_T^{\rm Ric})^{-1},
 \label{eq:lq-mobius}\\
 K_T(S)&=R^{-1}B^\top A^{-\top}(\mathscr P_T(S)+Q).
 \label{eq:lq-first-gain}
\end{align}
For $\mathscr P_\infty=P_sX_s^{-1}$ and
$K_\infty=R^{-1}B^\top A^{-\top}(\mathscr P_\infty+Q)$,
\begin{equation}\label{eq:lq-doubled-rate}
 \|\mathscr P_T(S)-\mathscr P_\infty\|+
 \|K_T(S)-K_\infty\|\le Ce^{-2\gamma T}.
\end{equation}
Uniformly for $S,\widetilde S$ in compact subsets of $S\preceq0$,
\begin{align}
 &\|\mathscr P_T(S)-\mathscr P_T(\widetilde S)\|
 +\|K_T(S)-K_T(\widetilde S)\|\nonumber\\
 &\qquad\le Ce^{-2\gamma T}\|S-\widetilde S\|.
 \label{eq:lq-terminal-hessian-forgetting}
\end{align}
Thus, when two terminal rewards differ only through their Hessians, the
induced Riccati matrices and initial gains differ by $O(e^{-2\gamma T})$.
Theorem~\ref{thm:terminal-reward} gives the one-sided rate for general
terminal-gradient perturbations.  The Riccati convergence itself is classical
\citep{caines1970discrete}; here it is derived explicitly from the same scaled
graph certificate.
\end{corollary}

\begin{proof}
Strict concavity follows from $R\succ0$, $Q\succeq0$, and $S\preceq0$.
Write $z_0=V_sa+V_uD_u^{-T}\eta$ and
$z_T=V_sD_s^Ta+V_u\eta$.  The terminal graph gives
$C_sD_s^Ta+C_u\eta=0$, hence
$\eta=-C_u^{-1}C_sD_s^Ta$; substitution at $t=0$ proves
\eqref{eq:lq-mobius}.  The costate identity
$p_1=A^{-\top}(p_0+Qx_0)$ proves \eqref{eq:lq-first-gain}.
The dichotomy gives $\|E_T^{\rm Ric}\|\le Ce^{-2\gamma T}$.  Moreover,
$(X_s-X_uE_T^{\rm Ric})^{-1}$ is uniformly bounded: for large $T$ this
follows from a Neumann argument around $X_s$, and the finitely many remaining
horizons follow from boundary-value uniqueness (uniformly on the stated
compact sets).  Therefore
\[
 \mathscr P_T-\mathscr P_\infty
 =(\mathscr P_\infty X_u-P_u)E_T^{\rm Ric}
   (X_s-X_uE_T^{\rm Ric})^{-1},
\]
which proves \eqref{eq:lq-doubled-rate}.  The map
$S\mapsto C_u^{-1}C_s$ is uniformly Lipschitz on the stated compact sets,
giving \eqref{eq:lq-terminal-hessian-forgetting}.
\end{proof}

\begin{corollary}[Nonlinear terminal condition with a nonpositive reference Hessian]\label{cor:original-concave}
Assume the dynamics and running cost are the LQ data of
Theorem~\ref{thm:lq}, translated to the stationary reference, and choose a
compact $U$ with $0\in\operatorname{int}U$.  If the terminal reward satisfies
the hypotheses of Theorem~\ref{thm:main}, with
$\pinf=r_x(\xinf,y_\infty)$ and
\[
       r_{xx}(\xinf,y_\infty)\preceq0 ,
\]
then the original Pontryagin problem
\[
      x_0=x_{\rm in},\qquad p_T=r_x(x_T,y)
\]
has the branch, explicit radius, and sensitivity estimates of
Theorem~\ref{thm:main}, with $T_0=1$.
\end{corollary}
\begin{proof}
Apply Theorem~\ref{thm:lq} with
$S=r_{xx}(\xinf,y_\infty)$, then apply Theorem~\ref{thm:main}.
\end{proof}

\section{Numerical illustration}\label{sec:numerics}
The finite-dimensional computations below evaluate the stated certificates
and nonlinear checks.  All coordinates in the example are nondimensional.
We use the coupled stabilizable data
\begin{equation}\label{eq:num-data}
\begin{aligned}
   A=\begin{pmatrix}1.5&0.4\\0&0.6\end{pmatrix},\qquad
   B&=\begin{pmatrix}1\\0.5\end{pmatrix},\\
   Q=\begin{pmatrix}1&0.3\\0.3&0.6\end{pmatrix},\qquad
   R&=1,
\end{aligned}
\end{equation}
with rectangular input.  Here $\lambda_{\min}(Q)\approx0.439$ and
$\|[A,Q]\|_2\approx0.334$.  Since $1.5$ is the only eigenvalue of $A$ with
modulus at least one and
$\sigma_{\min}([A-1.5I\ \ B])\approx0.980$, the PBH test gives
stabilizability, so Theorem~\ref{thm:lq} applies.  The matrix \eqref{eq:lq-M}
has eigenvalue moduli $\{0.367,0.556,1.798,2.723\}$, spectral gap
$\gamma\approx0.587$, and, for unit-norm eigenvector bases,
$\sigma_{\min}(X_s)\approx0.188$ and $\sigma_{\min}(P_u)\approx0.325$.
Thus the canonical graph transversality condition holds with an explicit
numerical margin.  For shifted
rows we take
\[
   S_{\rm term}=-\begin{pmatrix}0.8&0.2\\0.2&0.5\end{pmatrix}\preceq0,
\]
so Theorem~\ref{thm:lq} applies.  The nonlinear terminal reward is
\[
\begin{aligned}
 r_{\rm nl}(x)&=\tfrac12x^\top S_{\rm term}x
                 -(v_{\rm term}^\top x)^3,\\
 v_{\rm term}&=2^{-1/2}(1,1)^\top .
\end{aligned}
\]
Direct evaluation gives
$\lambda_{\max}(r_{{\rm nl},xx})\le-0.106$ on
$\|x\|_\infty\le0.08$.  For the nonlinear test we perturb the dynamics to
$\Fbar(x,u)=Ax+Bu+0.1\varphi(x)$, where
$\varphi(x)=(x_2^2,x_1x_2)$, and solve in the unit direction proportional to
$(0,0,-0.95,-0.30)$.  Table~\ref{tab:certificate-inputs} reports the scaled Green
test, quadratic remainder slopes, boundary conditioning, and the explicit
neighborhood and posterior tests.  The normalized remainder is
\[
 \widehat C(T,s)=\sup_{0\le t\le T}
 \frac{\|z_t-z_t^{\lin}\|}{w_T(t)^2s^2},
 \qquad \alpha=0.05<\gamma/2 .
\]
Across $T=20,40,80,160$, the scaled boundary condition number is at most
$7.725$, whereas a reliable column-norm lower bound for the unscaled matrix
reaches $7.56\times10^{69}$ at $T=160$.  Fits over
$s\in\{3\times10^{-4},10^{-3},3\times10^{-3},10^{-2}\}$ give remainder
slopes between $1.99987$ and $2.00713$ with $R^2\ge0.999997$.
The a posteriori test additionally uses $s=3\times10^{-2}$ and $s=10^{-1}$.
To test the decay caused by a terminal reward perturbation, we set
$r_s(x)=r_{\rm nl}(x)+s d^\top x$, with
$d=(-0.95,-0.30)^\top/\|(-0.95,-0.30)\|$, and evaluate
$|u_0^s-u_0^0|/s$.  Fits over $T=20,25,30,35$ for
$s\in\{0.003,0.01,0.03\}$ give rates $0.58706$--$0.58731$ with
$R^2>0.999999$, consistent with $\gamma=0.58664$.  The exact LQ graph calculation
for changing the terminal Hessian from $0$ to $S_{\rm term}$ gives rate
$1.17187$ with $R^2>0.999999$, consistent with $2\gamma=1.17328$.

\begin{table*}[t]
\caption{Representative numerical quantities.  Panel (a) evaluates the
direction-independent a priori bound at zero.  Panel (b) evaluates the
a posteriori bound at the linear predictor for $\Delta=s\widehat d$,
$s=10^{-2}$, $\|\widehat d\|_2=1$, and
$\|\widehat d\|_\infty=0.95358$.  Certificate norms are infinity norms;
only $\widehat d$ is normalized in the Euclidean norm.  Entries are upper
estimates obtained from a Neumann-series bound.}
\label{tab:certificate-inputs}
\centering\scriptsize
\renewcommand{\arraystretch}{1.00}
\begin{tabular}{@{}llrrrrr@{}}
\toprule
\multicolumn{7}{c}{(a) Direction-independent a priori bound evaluation}\\
\cmidrule(lr){1-7}
$T$ & row & $\mu_T$ & $c_T$ & $K$ & $\rho_T^\circ$
& $\rho_T^{\rm eval}$\\
\midrule
20  & can.   & $9.6705$ & $3.2807$ & $0.6000$
    & $2.6267{\times}10^{-2}$ & $2.1013{\times}10^{-2}$\\
20  & shift. & $9.6705$ & $3.1803$ & $8.4853$
    & $1.9160{\times}10^{-3}$ & $1.5328{\times}10^{-3}$\\
160 & can.   & $9.6706$ & $3.2807$ & $0.6000$
    & $2.6266{\times}10^{-2}$ & $2.1013{\times}10^{-2}$\\
160 & shift. & $9.6706$ & $3.1803$ & $8.4853$
    & $1.9160{\times}10^{-3}$ & $1.5328{\times}10^{-3}$\\
\bottomrule
\end{tabular}

\smallskip
\begin{tabular}{@{}llrrrrr@{}}
\toprule
\multicolumn{7}{c}{(b) A posteriori bound evaluation at $s=10^{-2}$}\\
\cmidrule(lr){1-7}
$T$ & row & $\beta_T^{\rm post}$ & $\omega_T^{\rm post}$
& $h_T^{\rm post}=\beta_T^{\rm post}\omega_T^{\rm post}$
& $r_{-,T}$ & $R_T^{\rm rep}$\\
\midrule
20  & can.   & $6.5069{\times}10^{-6}$ & $5.8023$
    & $3.7755{\times}10^{-5}$ & $6.5070{\times}10^{-6}$
    & $8.6176{\times}10^{-2}$\\
20  & shift. & $9.7658{\times}10^{-5}$ & $82.0572$
    & $8.0136{\times}10^{-3}$ & $9.8053{\times}10^{-5}$
    & $6.1423{\times}10^{-3}$\\
160 & can.   & $6.5069{\times}10^{-6}$ & $5.8024$
    & $3.7755{\times}10^{-5}$ & $6.5070{\times}10^{-6}$
    & $8.6175{\times}10^{-2}$\\
160 & shift. & $9.7658{\times}10^{-5}$ & $82.0578$
    & $8.0136{\times}10^{-3}$ & $9.8053{\times}10^{-5}$
    & $6.1423{\times}10^{-3}$\\
\bottomrule
\end{tabular}

\smallskip
\begin{minipage}[t]{.48\textwidth}
\centering
(c1) Finite-horizon boundary diagnostics

\smallskip
\begin{tabular}{@{}rccc@{}}
\toprule
$T$ & $\|\widehat\Gamma_T-\widehat\Gamma_\infty\|_2$
& $\sigma_{\min}$ (can.) & $\max\kappa_2$\\
\midrule
20 & $5.4\times10^{-6}$ & $0.1879$ & $7.72$\\
40 & $4.3\times10^{-11}$ & $0.1879$ & $7.72$\\
80 & $2.8\times10^{-21}$ & $0.1879$ & $7.72$\\
160 & $1.2\times10^{-41}$ & $0.1879$ & $7.72$\\
\bottomrule
\end{tabular}
\end{minipage}
\hfill
\begin{minipage}[t]{.48\textwidth}
\centering
(c2) Nonlinear diagnostics

\smallskip
\begin{tabular}{@{}lcc@{}}
\toprule
diagnostic & canonical & shifted\\
\midrule
fitted remainder slope & $1.99987$ & $2.00713$\\
a posteriori test satisfied & $24/24$ & $20/24$\\
\bottomrule
\end{tabular}
\end{minipage}
\end{table*}

For the finite-dimensional residual
$F_T(z;\Delta)=L_Tz-E\Delta+N_T(z)$, where
$N_T(0)=DN_T(0)=0$, we also evaluate an explicit direction-independent
Newton--Kantorovich data radius.  All operator norms in the following
calculation are ordinary trajectory infinity norms.  With
$\mu_T=\|L_T^{-1}\|_\infty$, $c_T=\|L_T^{-1}E\|_\infty$, and
$\|DN_T(z)-DN_T(w)\|_\infty\le K\|z-w\|_\infty$, the condition
\begin{equation}\label{eq:numerical-data-radius}
 \|\Delta\|_\infty<\rho_T:=\frac{1}{2\mu_TKc_T}
\end{equation}
places the linear predictor in the Kantorovich regime.  Indeed, if
$d_T^{\lin}=c_T\|\Delta\|_\infty$ and
$a_T=\mu_TK d_T^{\lin}<1/2$, then a Neumann bound at the
linear predictor gives
\begin{align*}
 \beta_T&\le\frac{\mu_TK(d_T^{\lin})^2}{2(1-a_T)},\qquad
 \omega_T\le\frac{\mu_TK}{1-a_T},\\
 \beta_T\omega_T&\le
 \frac12\left(\frac{a_T}{1-a_T}\right)^2<\frac12 .
\end{align*}
Thus Proposition~\ref{prop:posterior}, in the ordinary infinity norm,
applies.  Using $80\%$ of this
open threshold gives uniform evaluated radii $2.10\times10^{-2}$ for the
canonical row and $1.53\times10^{-3}$ for the nonlinear shifted row over the
four tested horizons.  These radii are valid for every endpoint direction,
unlike the directional
sampling.  The a posteriori test of Proposition~\ref{prop:posterior} succeeds in
44 of 48 runs: all canonical runs and all shifted runs with $s\le0.03$ pass;
the four shifted $s=0.1$ runs fail the sufficient inequality and lie outside
the stated concavity box; they are included only to examine the sufficient
test outside the stated concavity range.  All shifted runs with
$s\le0.03$ remain inside that box.  Failure of the test does not imply
nonexistence.  At $s=10^{-2}$, the shifted posterior
reported admissible uniqueness radius is approximately $6.14\times10^{-3}$.
These are floating-point contraction checks using inverse-norm bounds with a
Neumann correction,
not interval-verified computer-assisted proofs.

Table~\ref{tab:certificate-inputs} reports the numerical inputs behind the
reported radii.  Here $K$ is a global infinity-norm Lipschitz bound for the
Jacobian of the polynomial residual, so Proposition~\ref{prop:posterior}
allows $R_{\rm tube}=+\infty$.  Also,
\[
 \rho_T^\circ=\rho_T\ \text{in \eqref{eq:numerical-data-radius}},\qquad
 \rho_T^{\rm eval}=0.8\rho_T^\circ .
\]
Panel (b) reports the directly evaluated posterior quantities
$\beta_T^{\rm post},\omega_T^{\rm post}$.  There $r_{-,T}$ bounds the
distance from the linear predictor to an exact zero, while
\begin{equation}\label{eq:reported-uniqueness-radius}
 R_T^{\rm rep}=\tfrac12(r_{-,T}+1/\omega_T^{\rm post})
\end{equation}
is one concrete radius strictly inside the admissible uniqueness interval,
not a maximal radius.  The inverse residual in the displayed rows is below
$2.1\times10^{-15}$.  The directional posterior test can succeed outside the
smaller direction-independent a priori ball because it is centered at the computed
predictor for that particular direction.

The additional sweep $S(\theta)=S_{\rm term}+\theta I$ detects the value at
which limiting transversality fails, near $\theta_*\approx1.23347$; the least singular values there
are $4.67\times10^{-16}$ for $\widehat\Gamma_\infty$ and
$3.86\times10^{-12}$ for $T=20$.  The terminal perturbation data, remainder slopes,
radii, and pass counts are summarized above.

\section{Concluding remarks}\label{sec:conclusion}
The scaled boundary matrix gives a Green estimate that separates both endpoint
effects.  Its one-sided terminal estimate shows that an error in the terminal
reward gradient changes the initial stationary control and the gradient with
respect to the initial state by $O(e^{-\alpha_{\rm ter}T})$, with an explicit
local constant.  The same analysis gives horizon-uniform branches, quadratic
remainders, and a priori and a posteriori certificates.  For LQ systems with
invertible $A$, stabilizable $(A,B)$, $Q\succ0$, $R\succ0$, and a nonpositive
terminal Hessian, the certificate holds from $T=1$, and sensitivity to that
Hessian decays at rate $O(e^{-2\gamma T})$.  Under strict concavity these
stationary solutions are Bellman optimal.  Establishing horizon-uniform
second-order sufficient conditions for the nonlinear problem remains open.

\section*{Data and code availability}
No external data were used; all numerical values follow from the equations and
parameters in this article.

\section*{Declaration of generative AI and AI-assisted technologies in the writing process}
The authors used OpenAI Codex in ChatGPT Work for language editing and to draft
parts of the numerical code.  They independently verified the analytic
Jacobian and numerical results and take full responsibility for this article.

\begingroup
\fontsize{7.8}{8.15}\selectfont
\setlength{\bibsep}{0pt}

\endgroup
\end{document}